\documentclass[12pt,notitlepage]{amsart}
\usepackage{latexsym,amsfonts,amssymb,amsmath,amsthm}
\usepackage{color}
\usepackage{esint}
\usepackage[inner=1.0in,outer=1.0in,bottom=1.0in, top=1.0in]{geometry}

\newcommand{\mz}{\ensuremath{\mathbb Z}}
\newcommand{\mr}{\ensuremath{\mathbb R}}
\newcommand{\mh}{\ensuremath{\mathbb H}}

\newcommand{\half}{\ensuremath{ \frac{1}{2}}}

\newcommand{\intR}{\int_{-\infty}^{\infty}}

\theoremstyle{plain}
	\newtheorem{mytheo}{Theorem} [section]

	\newtheorem{mycoro}[mytheo]{Corollary}
     \newtheorem{mylemma}[mytheo]{Lemma}

\theoremstyle{remark}

\numberwithin{equation}{section}
\begin{document}

%
 \author{Matthew P. Young}
 \address{Department of Mathematics \\
 	  Texas A\&M University \\
 	  College Station \\
 	  TX 77843-3368 \\
 		U.S.A.}
 \email{myoung@math.tamu.edu}
 \thanks{This material is based upon work supported by the National Science Foundation under agreement No. DMS-1401008.  Any opinions, findings and conclusions or recommendations expressed in this material are those of the authors and do not necessarily reflect the views of the National Science Foundation.  
 }

%
 \title{A note on the sup norm of Eisenstein series}
 \maketitle

 \section{Introduction}
Let $E(z,s)$ be the usual real-analytic Eisenstein series for the group $\Gamma = PSL_2(\mz)$. 
 \begin{mytheo}
  Let $K$ be a fixed compact subset of $\mh$.  Then for $T \geq 1$,
  \begin{equation}
   \max_{z \in K} |E(z, 1/2 + iT)| \ll_{K, \varepsilon} T^{3/8 + \varepsilon}.
  \end{equation}
 \end{mytheo}
For comparison, Iwaniec and Sarnak showed, for $u_j$ a Hecke-Maass cusp form of spectral parameter $t_j$, that
\begin{equation}
 \max_{z \in K} |u_j(z)| \ll_K t_j^{5/12 + \varepsilon}.
\end{equation}
The sup norm problem has been an active area and now there exist non-trivial estimates for cusp forms of large level, on higher rank groups, and for half-integral weight forms \cite{BlomerHolowinsky} \cite{Templier1} \cite{HarcosTemplier1} \cite{HarcosTemplier2} \cite{Templier2} \cite{HRR}
\cite{BlomerPohl} \cite{BlomerMaga} \cite{Marshall} \cite{Kiral}.  Nevertheless, the basic estimate for $\Gamma = PSL_2(\mz)$ in the eigenvalue aspect has not been improved.  The case of Eisenstein series seems to have been largely neglected up to now, at least for the sup norm problem, but not for some other norms: Luo and Sarnak proved QUE for Eisenstein series \cite{LuoSarnakQUE}, Spinu estimated the $L^4$ norm \cite{Spinu}, and the author has investigated QUE for geodesic restrictions \cite{Young}.
The Eisenstein series case is similar in some ways to the cuspidal case, but has some technical problems because of the constant term in the Fourier expansion.  Actually, the main impetus for this note was the realization that one can choose an efficient amplifier for the Eisenstein series, which leads to the improved exponent compared to the cusp form case (see \cite[Remark 1.6]{IwaniecSarnak}).

 \section{Notation and summary of the results of Iwaniec and Sarnak}
For $\text{Re}(s) > 1$, let
 \begin{equation}
 \label{eq:EisensteinDef}
  E(z,s) = \sum_{\gamma \in \Gamma_{\infty} \backslash \Gamma} (\text{Im}(\gamma z))^s.
 \end{equation}
As shorthand, we sometimes write $E_t(z) = E(z, 1/2 + it)$.  The Fourier expansion states
\begin{equation}
\label{eq:Fourier}
 \zeta^*(2s) E(z,s) = \zeta^*(2s) y^s + \zeta^*(2(1-s)) y^{1-s} + 2 \sqrt{y} \sum_{n \neq 0} \tau_{s-\frac12}(n) e(nx) K_{s-\half}(2 \pi |n| y),
\end{equation}
where $\tau_{w}(n) = \sum_{ab = |n|} (a/b)^w$, and
\begin{equation}
 \zeta^*(2s) = \pi^{-s} \Gamma(s)
\zeta(2s).
\end{equation}
The Fourier expansion implies the functional equation $\zeta^*(2s) E(z,s) = \zeta^*(2(1-s)) E(z, 1-s)$.
Specializing to $s=1/2 + it$, and setting $\varphi(s) = \zeta^*(2(1-s))/\zeta^*(2s)$, one obtains
\begin{equation}
 E_t(z) = y^{1/2 + it} + \varphi(1/2+it) y^{1/2-it} + \frac{2 \sqrt{y}}{\zeta^*(1+2it)} \sum_{n \neq 0} \tau_{it}(n) e(nx) K_{it}(2 \pi |n| y).
\end{equation}

 Let $\alpha_n$ be an arbitrary sequence of complex numbers.  The main technical result proved by Iwaniec and Sarnak \cite[(A.12)]{IwaniecSarnak} is
 \begin{multline}
 \label{eq:IwaniecSarnakBound}
  \sum_{T \leq t_j \leq T+1} \Big| \sum_{n \leq N} \alpha_n \lambda_j(n) \Big|^2 |u_j(z)|^2 + \int_T^{T+1} \Big| \sum_{n \leq N} \alpha_n \tau_{it}(n) \Big|^2 |E_t(z)|^2 dt  
  \\
  \ll
  T^{\varepsilon} \Big(T\sum_{n \leq N} |\alpha_n|^2 + T^{1/2} (N + N^{1/2} y) \Big(\sum_{n \leq N} |\alpha_n| \Big)^2 \Big).
 \end{multline}
Here $\lambda_j(n)$ are the Hecke eigenvalues of the Hecke-Maass cusp forms $u_j$, scaled so the Ramanujan-Petersson conjecture is $|\lambda_j(n)| \leq \tau_0(n)$.
The implied constant is uniform in $y$ for $y \gg 1$.

We have two main problems to overcome to obtain a bound on $E_t$.  The first problem is to relate a pointwise bound on $E_t$ to an integral bound of the type occuring in \eqref{eq:IwaniecSarnakBound}.  We are able to accomplish this by modifying a method of Heath-Brown \cite[Lemma 3]{HeathBrown}.  This shows, roughly, that $|E_T(z)|^2 \lessapprox \int_{T-1}^{T+1} |E_t(z)|^2 dt$ (see Corollary \ref{coro:EpointwiseIntegral} below for the true result).  
Normally one constructs an amplifier to be large at a specified point.  Because of the above relationship between the integral of $|E_t|^2$ and the pointwise bound, we cannot simply choose $\alpha_n$ to be large at a single value of $t$.  Rather, we need the amplifier to be large on an interval of $t$'s of length $\gg T^{-\varepsilon}$.  This we accomplish with Lemma \ref{lemma:amplifier} below. 
 
 \section{Preliminary estimates}
 For purposes of comparison it is helpful to record the effects of trivially bounding the Eisenstein series using the Fourier expansion.  
Let $F(z,s) = E(z,s) - y^s - \varphi(s) y^{1-s}$.
\begin{mylemma}
\label{lemma:Fbound}
For $t \geq 1$, and $y \gg 1$, we have
 \begin{equation}
 \label{eq:FourierExpansionBound}
 F(z, 1/2 + it) \ll (t/y)^{1/2} \log^2 t + y^{1/2} t^{-1/3 + \varepsilon},
\end{equation}
and therefore,
\begin{equation}
\label{eq:FourierExpansionBound2}
 E(z, 1/2 + it) \ll y^{1/2} + (t/y)^{1/2} \log^2 t.
\end{equation}
\end{mylemma}
This is analogous to \cite[Proposition 6.2]{Templier2}.
\begin{proof}
Suppose that $t \geq 1$.  By Stirling's formula,
\begin{equation}
 F(z, 1/2 + it) \ll \frac{\sqrt{y}}{|\zeta(1+2it)|} \sum_{n=1}^{\infty} \tau_0(n) |K_{it}(2 \pi n y)| \cosh(\pi t/2).
\end{equation}
Next we need uniform bounds on the $K$-Bessel function which we extract from the uniform asymptotic expansions due to Balogh \cite{Balogh}:
\begin{equation}
\label{eq:KBesselBalogh}
 \cosh(\pi t/2) K_{it}(u) \ll
 \begin{cases}
  t^{-1/4} (t -u)^{-1/4}, \quad  &\text{if }   0 < u < t - C t^{1/3} \\
  t^{-1/3}, \quad &\text{if } |u-t| \leq C t^{1/3},
 \\ 
u^{-1/4} (u -t)^{-1/4} \exp\Big(- c (\frac{u}{t})^{3/2} \big(\frac{u-t}{t^{1/3}}\big)^{3/2} \Big), \quad  &\text{if }    u > t + C t^{1/3}.
 \end{cases}
\end{equation}
We break up the sum over $n$ according to the different pieces.  For instance, the range $2 \pi n y \leq \frac{t}{2}$ gives a bound
\begin{equation}
\label{eq:FourierExpansionBulkBound}
 \frac{\sqrt{y}}{|\zeta(1+2it)|} \sum_{n \ll t/y} \frac{\tau_0(n)}{t^{1/2}} \ll (t/y)^{1/2} \log^2 t,
\end{equation}
using $|\zeta(1+2it)|^{-1} \ll \log t$. 
Similarly, the range $|2\pi ny - t| \asymp \Delta$ with $t^{1/3} \ll \Delta \ll t$ gives
\begin{equation}
 \frac{\sqrt{y}}{|\zeta(1+2it)|}  (t\Delta)^{-1/4} \sum_{n = \frac{t}{2 \pi y} + O(\frac{\Delta}{ y})} \tau_0(n) \ll 
 \frac{\sqrt{y} \log t}{(t\Delta)^{1/4}}   
 \Big(\frac{\Delta}{y} \log t + t^{\varepsilon} \Big),
\end{equation}
using Shiu's bound on the divisor function in a short interval \cite{Shiu}.
These terms then give the same bound as \eqref{eq:FourierExpansionBulkBound}, plus an additional term that is
\begin{equation}
\label{eq:transitionbound}
 \ll y^{1/2} t^{-1/3 + \varepsilon}.
\end{equation}
It is easily checked that the transition range also leads to \eqref{eq:transitionbound}.  The range with $u > t + C t^{1/3}$ is even easier to bound because there is an additional exponential factor aiding the estimation. In all, this shows \eqref{eq:FourierExpansionBound}.  The bound \eqref{eq:FourierExpansionBound2} is immediate from \eqref{eq:FourierExpansionBound}.
\end{proof}

We also record some bounds on $F(z,s)$ and $E(z,s)$ valid not just for $\text{Re}(s) = 1/2$.  In principle, a similar approach to the proof of Lemma \ref{lemma:Fbound} should apply, however the author has been unable to find sharp estimates analogous to \eqref{eq:KBesselBalogh} for $K_{\sigma+it}$ with $\sigma > 0$, and so we present some weaker bounds that suffice for our later application.
\begin{mylemma}
 If $\text{Re}(s)  > 1$, then
  \begin{equation}
 \label{eq:Fbound}
 F(z,s) \ll_{\sigma} y^{1-\sigma},
\end{equation}
uniformly for $y \gg 1$, $ t \in \mr$.  If $\text{Re}(s) \geq 1/2 - \frac{c}{\log(2+|t|)}$, for a sufficiently small $c > 0$, then
\begin{equation}
 \label{eq:FboundCriticalStrip}
 F(z,\sigma + it) \ll_{\sigma, \varepsilon} (1+|t|)^{1+\varepsilon} y^{-\sigma}.
\end{equation}
Moreover, if $y \gg t^{1+\varepsilon}$, and $t$ is sufficiently large, then
\begin{equation}
\label{eq:FboundBeyondT}
 F(z,\sigma + it) \ll_{\sigma, \varepsilon} (yt)^{-100}.
\end{equation}
\end{mylemma}
\begin{proof}
We begin with \eqref{eq:Fbound}.  Using the original definition \eqref{eq:EisensteinDef}, we have
\begin{equation}
F(z,s) = - \varphi(s) y^{1-s} + \sum_{\substack{\gamma \in \Gamma_{\infty} \backslash \Gamma \\ \gamma \neq 1}} (\text{Im}(\gamma z))^s,
\end{equation}
and so by a trivial bound,
\begin{equation}
|F(z,s)| \leq  |\varphi(s)| y^{1-\sigma} +  \sum_{\substack{\gamma \in \Gamma_{\infty} \backslash \Gamma \\ \gamma \neq 1}} (\text{Im}(\gamma z))^\sigma
= |\varphi(s)| y^{1-\sigma} + (E(z,\sigma) - y^{\sigma}).
\end{equation} 
 By the Fourier expansion, and using the standard bound
\begin{equation} 
\label{eq:KBesselFixedOrderLargex}
 K_{\alpha}(x) \ll_{\alpha} x^{-1/2} e^{-x}, 
\end{equation} 
 for $x \gg 1$, we have
\begin{equation} 
\label{eq:Ezsigmabound}
 E(z,\sigma) - y^{\sigma} \ll_{\sigma} y^{1-\sigma} +  \sum_{n =1}^{\infty} \tau_{\sigma-\half}(n) n^{-1/2} \exp(-2 \pi n y) \ll_{\sigma} y^{1-\sigma},
\end{equation}
since the infinite sum in \eqref{eq:Ezsigmabound} is $\ll_{\sigma} \exp(-2\pi y)$.  

Furthermore, note that for $\sigma > 1$, we have
\begin{equation}
\varphi(\sigma + it) = \frac{\zeta^*(2-2\sigma - 2it)}{\zeta^*(2\sigma + 2it)} = \frac{\zeta^*(-1+2\sigma + 2it)}{\zeta^*(2\sigma + 2it)} \ll_{\sigma} \frac{\big|\Gamma(\sigma-\frac12+it)|}{\big|\Gamma(\sigma +it)|} \ll_{\sigma} (1 + |t|)^{-1/2}.
\end{equation}
 Combining the above estimates, we derive \eqref{eq:Fbound}.
 
 Next we show \eqref{eq:FboundCriticalStrip} and \eqref{eq:FboundBeyondT}.  By the Fourier expansion \eqref{eq:Fourier}, we have
 \begin{equation}
 \label{eq:FabsolutevaluewithFourier}
  |F(z,s)| = \Big|\frac{2 y^{1/2}}{\pi^{-s} \Gamma(s) \zeta(2s)} \sum_{n \neq 0} \tau_{s-\half}(n) e(nx) K_{s-\frac12} (2 \pi |n| y) \Big|.
 \end{equation}
We will momentarily show that if $t \gg 1$, then 
\begin{equation}
\label{eq:KBesselBound}
\frac{K_{\sigma + it}(y)}{\Gamma(\frac12 + \sigma + it)} \ll_{\sigma, \lambda}  y^{-\sigma} (t/y)^{\lambda},
\end{equation}
where $\lambda > 0$ may be chosen at will.  If $t \ll 1$, then \eqref{eq:KBesselFixedOrderLargex} gives an even stronger bound than \eqref{eq:KBesselBound} (one should note that the implied constant in \eqref{eq:KBesselFixedOrderLargex} depends continuously on $\alpha$).
Now we insert this bound into \eqref{eq:FabsolutevaluewithFourier}, giving
\begin{equation}
|F(z,\sigma + it)| \ll_{\sigma,\lambda} \frac{y^{1-\sigma-\lambda} t^{\lambda}}{|\zeta(2\sigma + 2it)|} \sum_{n=1}^{\infty} n^{\sigma-\frac12} \tau_0(n) n^{\frac12-\sigma-\lambda} \ll_{\sigma,\lambda} \frac{y^{1-\sigma} (t/y)^{\lambda}}{|\zeta(2\sigma + 2it)|},
\end{equation}
where we hereby assume $\lambda > 1$ to ensure convergence.  If $y \gg t^{1+\varepsilon}$ we may take $\lambda$ very large to give \eqref{eq:FboundBeyondT}.  In any case, we may take $\lambda = 1 + \varepsilon$ with $\varepsilon > 0$ small, giving \eqref{eq:FboundCriticalStrip}. 
 
Finally, we show \eqref{eq:KBesselBound}.  By \cite[17.43.18]{GR}, we have
\begin{equation}
\frac{4 K_{\nu}(y)}{\Gamma(\frac12 + \nu)} = \frac{1}{2 \pi i} \int_{(\delta)} (y/2)^{-w} \frac{\Gamma(\frac{w+\nu}{2}) \Gamma(\frac{w-\nu}{2})}{\Gamma(\frac12 + \nu)} dw,
\end{equation} 
where $\delta > |\text{Re}(\nu)|$.
By Stirling, we have with $w = \delta + iv$ and $\nu = \sigma + it$, that
\begin{equation}
\frac{\Gamma(\frac{w+\nu}{2}) \Gamma(\frac{w-\nu}{2})}{\Gamma(\frac12 + \nu)} \ll_{\sigma,\delta} t^{-\sigma} (1 + |v+t|)^{\frac{\delta + \sigma-1}{2}} (1 + |v-t|)^{\frac{\delta - \sigma-1}{2}}  \exp(-\tfrac{\pi}{2} q(v,t)),
\end{equation}
 where $q(v,t) = 0$ for $|v| \leq |t|$, and $q(v,t) = |v-t|$ for $|v| \geq |t|$.  Therefore, 
\begin{equation}
\frac{K_{\sigma+it}(y)}{\Gamma(\frac12 + \sigma + it)} \ll y^{-\delta} t^{-\sigma} \intR   (1 + |v+t|)^{\frac{\delta + \sigma-1}{2}} (1 + |v-t|)^{\frac{\delta - \sigma-1}{2}}  \exp(-\tfrac{\pi}{2} q(v,t)) dv.
\end{equation} 
By symmetry, suppose $v,t \geq 0$.  The part of the integral with $0 \leq v \leq t$ gives a bound
\begin{equation}
\label{eq:BesselBoundIntegralPart}
y^{-\delta}  t^{-\sigma + \frac{\delta + \sigma-1}{2}} \int_0^{t} (1 + v)^{\frac{\delta - \sigma-1}{2}} dv \ll 
y^{-\delta} t^{-\sigma +  \frac{\delta + \sigma-1}{2} + \frac{\delta-\sigma+1}{2} }
= y^{-\delta} t^{\delta -\sigma}.
\end{equation}  
We shall choose $\delta = \sigma + \lambda$ with $\lambda > 0$ arbitrary, which gives a bound consistent with \eqref{eq:KBesselBound}.  The part of the integral with $v \geq t$ is easier to bound, aided by the exponential decay that is present in this range.  Firstly, one may note that the range with $t \leq v \leq 2t$ gives precisely the same bound as \eqref{eq:BesselBoundIntegralPart}, even if we only use the very crude bound $q(v,t) \geq 0$.  For the range with $v \geq 2t$, then we have the bound
\begin{equation}
y^{-\delta} t^{-\sigma} \int_{v \geq 2t} v^{\delta-1} \exp(-\tfrac{\pi}{2}(v-t)) dv \ll_{\delta} y^{-\delta}  \exp(-\pi t/2). \qedhere
\end{equation}
\end{proof}

\section{A pointwise bound via an integral bound}
\begin{mylemma}
\label{lemma:FpointwiseIntegral}
 Suppose $y, T \gg 1$ and $y \ll T^{100}$.  Then
 \begin{equation}
 \label{eq:FpointwiseIntegral}
  |F(z,1/2 + iT)|^2 \ll 
   \frac{\log^5 T}{y}
 + \log^5 T \int_{|r| \leq 4 \log T} |F(z, 1/2 + iT + ir)|^2 dr.
 \end{equation}
\end{mylemma}
The proof is analogous to that of Heath-Brown \cite{HeathBrown}, but some modifications are necessary since the Eisenstein series is not a Dirichlet series, and $z$ is an additional parameter to control.   The key idea is to work with $F(z,s)$, and to deduce analogous results for $E(z,s)$ only at the end (see Corollary \ref{coro:EpointwiseIntegral} below).  The logarithmic powers could be reduced with a little more work, if it were necessary.

\begin{proof}
 Note that $F(z,s)$ satisfies the functional equation $\zeta^*(2s) F(z,s) = \zeta^*(2(1-s)) F(z,1-s)$, just as the Eisenstein series does.  Furthermore, $\zeta^*(2s)F(z,s)$ is entire, and so $F(z,s)$ is analytic for $\text{Re}(s) \geq 1/2$.

 Suppose that $\text{Re}(s) \in [1/4, 3/4]$, and let
 \begin{equation}
 I = \frac{1}{2 \pi i} \int_{(2)} F(z, s+w)^2 \frac{\exp(w^2)}{w} dw.
 \end{equation}
By \eqref{eq:Fbound}, we have that $I = O(y^{-1})$, uniformly in the stated range of $s$, and for $y \gg 1$.
Suppose that $0 < \delta < 1/8$, and $s = 1/2 + \delta + iT$.  Then by shifting the contour of integration to $\text{Re}(w) = - \delta$, we obtain
\begin{equation}
  F(z, 1/2 + \delta + iT)^2 + \frac{1}{2 \pi i} \int_{(-\delta)} F(z, 1/2 + \delta + iT + w)^2 \frac{\exp(w^2)}{w} dw = I = O(y^{-1}).
\end{equation}
By bounding the integral trivially with absolute values, we have
\begin{equation}
\label{eq:Fshiftedbound}
 |F(z, 1/2 + \delta + iT)|^2 \ll y^{-1} + \intR \frac{\exp(-v^2)}{|-\delta+iv|}  |F(z, 1/2 + iT + iv)|^2 dv.
\end{equation}

Alternatively, we have by Cauchy's theorem that
\begin{equation}
 F(z, 1/2 + iT)^2 = \frac{1}{2 \pi i} \ointctrclockwise F(z, 1/2 + iT + u)^2 \frac{\exp(u^2)}{u} du,
\end{equation}
where the contour of integration is a small loop around $u=0$.  We choose the contour to be a rectangle with corners $\pm \delta \pm 2i\log T$, where $\delta = \frac{c}{\log T}$ with $c > 0$ small enough to ensure that $\zeta(1 + 2iT + 2u) \gg (\log T)^{-1}$ inside the contour (by the standard zero-free region of $\zeta$).  Using \eqref{eq:FboundCriticalStrip}, it follows that the top and bottom sides of this integral are bounded by
\begin{equation}
\label{eq:topandbottom}
 y^{-1} \exp(- \log^2 T),
\end{equation}
which additionally uses the fact that $y^{-1/2 - \text{Re}(u)} \asymp y^{-1/2}$ since $\log y \ll  \log T$.

For the side of the rectangle with $\text{Re}(u) = -\delta$, we change variables $u \rightarrow - u$ and apply the functional equation of $F$, which gives
\begin{equation}
 F(z, 1/2 -\delta + iT  - iv) = \frac{\zeta^*(1+2\delta - 2iT + 2iv)}{\zeta^*(1-2\delta + 2iT - 2iv)} F(z, 1/2 + \delta -iT + iv).
\end{equation}
Using standard bounds on the zeta function, we derive
\begin{equation}
 \frac{\zeta^*(1+2\delta - 2iT + 2iv)}{\zeta^*(1-2\delta + 2iT - 2iv)} \ll \log^2 T.
\end{equation}
Therefore, we conclude
\begin{equation}
\label{eq:Fshiftedbound2}
 |F(z, 1/2 + iT)|^2 \ll y^{-1} \exp(- \log^2 T) + (\log T)^4 \int_{|v| \leq 2\log T} \frac{\exp(-v^2)}{|\delta + iv|} |F(z, 1/2 + \delta + iT + iv)|^2 dv.
\end{equation}

Inserting \eqref{eq:Fshiftedbound} into \eqref{eq:Fshiftedbound2}, we derive
\begin{multline}
 |F(z, 1/2 + iT)|^2 \ll y^{-1} \exp(- \log^2 T) 
 \\
 + 
 (\log T)^4 \int_{|v| \leq 2\log T} \frac{\exp(-v^2)}{|\delta + iv|} \Big[y^{-1} + \intR \frac{\exp(-r^2)}{|-\delta+ir|}  |F(z, 1/2 + iT + iv+ ir)|^2 dr \Big] dv.
\end{multline}
The inner $r$-integral can be safely truncated at $|r| \leq 2\log T$ without introducing a new error term.  Changing variables $r \rightarrow r- v$, and extending the $r$-integral to $|r| \leq 4 \log T$ by positivity, we derive
 \begin{multline}
 |F(z, 1/2 + iT)|^2 \ll 
 y^{-1} \log^5 T
 \\
 + 
 (\log T)^4  \int_{|r| \leq 4 \log T} |F(z, 1/2 + iT + ir)|^2 \Big(\int_{|v| \leq 2\log T} \frac{\exp(-v^2)}{|\delta + iv|}  \frac{\exp(-(r-v)^2)}{|-\delta+i(r-v)|}   dv \Big) dr.
\end{multline}
Using Cauchy-Schwarz, and the obvious estimate
\begin{equation}
 \intR \frac{1}{|\delta + iv|^2} dv = \frac{\pi}{\delta} \ll \log T, 
\end{equation}
we may bound the inner $v$-integral by $O(\log T)$.  Putting everything together gives
the desired bound \eqref{eq:FpointwiseIntegral}.
\end{proof}

Using $E(z,s) = y^s + \varphi(s) y^{1-s} + F(z,s)$, and Cauchy-Schwarz, we derive a corresponding result for the Eisenstein series itself.
\begin{mycoro}
\label{coro:EpointwiseIntegral}
 Suppose $y, T \gg 1$.  Then
 \begin{equation}
 \label{eq:EpointwiseIntegral}
  |E(z,1/2 + iT)|^2 \ll 
   y \log^6 T   
 + \log^5 T \int_{|r| \leq 4 \log T} |E(z, 1/2 + iT + ir)|^2 dr.
 \end{equation}
\end{mycoro}

\section{A lower bound for the amplifier}
Let $w$ be a fixed, compactly-supported function on the positive reals, with $\intR w(t) dt \neq 0$.  Define
\begin{equation}
 A_N(t,r) = \sum_{n=1}^{\infty} w(n/N) \tau_{it}(n) \tau_{ir}(n).
\end{equation}
\begin{mylemma}
\label{lemma:amplifier}
 Suppose that $\log N \gg (\log T)^{2/3 + \delta}$, and $t,r = T + O( (\log N)^{-1-\delta})$, for some fixed $\delta > 0$.  Then 
 \begin{equation}
  A_N(t,r)=   \frac{\widetilde{w}(1) N \log N }{\zeta(2) |\zeta(1+2iT)|^2} (1 + o(1)).
 \end{equation}
\end{mylemma}
Fouvry, Kowalski, and Michel \cite[Lemma 2.4]{FKM} prove a result with a similar conclusion, but their method requires $N \gg T^{3}$, while here we eventually will want $N = T^{1/4}$.

\begin{proof}
 Taking a Mellin transform, and using a well-known identity of Ramanujan \cite[(15)]{Ramanujan},
 we derive
 \begin{equation}
  A_N(t,r) = \frac{1}{2 \pi i} \int_{(2)} N^s \widetilde{w}(s) \zeta^{-1}(2s) \zeta(s+it+ir) \zeta(s+it-ir) \zeta(s-it+ir) \zeta(s-it-ir) ds.
 \end{equation}
Next we move the contour to the left, to one along the straight line segments  $L_1, L_2, L_3$ defined by 
$L_1 = \{1 - \frac{c}{(\log T)^{2/3}} + it: |t| \leq 100 T \}$, 
$L_2 = \{ it :  |t| \geq 100 T \}$, and the short horizontal segments $L_3 = \{ \sigma \pm 100 iT: 1-\frac{c}{(\log T)^{2/3}} \leq \sigma \leq 1 \}$.  The integrals along the line segments $L_2$ and $L_3$ are trivially bounded by $O(T^{-100})$ by the rapid decay of $\widetilde{w}$.  The new line $L_1$ gives an amount that is
\begin{equation}
\label{eq:amplifierErrorTerm}
 \ll N \log^{8/3} T \exp\Big(- c \frac{\log N}{(\log T)^{2/3}}\Big) \ll \frac{N}{(\log T)^{100}},
\end{equation}
using the Vinogradov-Korobov bound $|\zeta(\sigma +i t)| \ll (\log |t|)^{2/3}$ for $t \gg 1$ and $1-\frac{c}{(\log |t|)^{2/3}} \leq \sigma \leq 1$ (see \cite[Corollary 8.28]{IK}).  Here what is important is that $\sigma $ can be taken to be $1 - O((\log |t|)^{-1+\delta})$, for some small $\delta > 0$.

We need to analyze the residues of the poles.  Temporarily assume that $t \neq \pm r$.  
The poles at $s = 1 + it + ir$ and $s = 1-it -ir$ have very small residues from the rapid decay of $\widetilde{w}(s)$.  The residue at $s= 1 +ir-it$ contributes
\begin{equation}
R_1 = \frac{N^{1+ir-it} \widetilde{w}(1+ir-it) }{\zeta(2(1+ir-it))} \zeta(1+2ir) \zeta(1+2ir-2it) \zeta(1-2it).
\end{equation}
By symmetry the residue at $s=1-ir + it$, say $R_2$, is the same as $R_1$ but with $r$ and $t$ interchanged.
Let us write $r = t + \eta$ (by assumption, $\eta = O((\log N)^{-1-\delta}$), so 
\begin{equation}
 R_1 = \frac{N^{1+i\eta} \widetilde{w}(1+i\eta) }{\zeta(2(1+i\eta))} \zeta(1+2i\eta) \zeta(1+2it + 2i \eta)  \zeta(1-2it).
\end{equation}
By simple Taylor approximations, we have
\begin{equation}
\label{eq:R1simplepart}
 \frac{N^{1+i\eta} \widetilde{w}(1+i\eta) }{\zeta(2(1+i\eta))} \zeta(1+2i\eta) =\frac{\widetilde{w}(1) N \log N }{2 \zeta(2)} \Big( \frac{1}{i \eta \log N} + 1 + O(|\eta| \log N) \Big),
\end{equation}
and 
by the Vinogradov-Korobov bound $\frac{\zeta'}{\zeta}(1 + 2it) \ll (\log |t|)^{2/3 + \varepsilon}$ (see \cite[Theorem 8.29]{IK}), 
we have
\begin{equation}
\label{eq:R1zetapart}
 \zeta(1+2it + 2i \eta) \zeta(1-2it) = |\zeta(1+2it)|^2 (1 + O(|\eta| (\log T)^{2/3 + \varepsilon})).
\end{equation}
Combining \eqref{eq:R1simplepart} and \eqref{eq:R1zetapart}, we derive
%
\begin{equation}
\label{eq:R1estimate}
 R_1 = \frac{N\widetilde{w}(1) |\zeta(1+2it)|^2 \log N}{2\zeta(2)}  \Big[ 1+ \frac{1}{i \eta \log N}   + O\Big(|\eta| \log N + \frac{(\log T)^{2/3+\varepsilon}}{\log N} + |\eta| (\log T)^{2/3+\varepsilon}\Big) \Big].
\end{equation}
The conditions $\eta \ll (\log N)^{-1-\delta}$, $(\log N) \gg (\log T)^{2/3+\delta}$ are enough to imply that the error term in \eqref{eq:R1estimate} is $o(1)$.
Similarly,
\begin{equation}
 R_2 = \frac{N\widetilde{w}(1) |\zeta(1+2it)|^2 \log N}{2\zeta(2)}  \Big[1 + \frac{1}{-i \eta \log N}   + o(1)
 \Big],
\end{equation}
and therefore,
\begin{equation}
 R_1 + R_2 = \frac{N\widetilde{w}(1) |\zeta(1+2it)|^2 \log N }{\zeta(2)} (1 + o(1)
 ). \qedhere
\end{equation}
{\bf Remark.}  If we had moved the contour to the line $\sigma = 1/2$, then instead of \eqref{eq:amplifierErrorTerm} we would have obtained an error term of size $O(N^{1/2} T^{1/3+\varepsilon})$ using Weyl's subconvexity bound.  This is $o(N)$ for $N \gg T^{2/3+\varepsilon}$, which is far from our desired choice of $N = T^{1/4}$.
\end{proof}

\section{Completion of the proof}
By Corollary \ref{coro:EpointwiseIntegral}, we have that
\begin{equation}
 |E(z, 1/2 + iT)|^2 \ll y \log^6 T   
 + \log^5 T \int_{|r| \leq 4 \log T} |E(z, 1/2 + iT + ir)|^2 dr.
\end{equation}
On the right hand side above, we dissect the integral into subintervals, each of length $\asymp (\log T)^{-2}$, say.  Let $U$ be one of these intervals, and choose a point $t_U \in U$.  Then by Lemma \ref{lemma:amplifier}, we have
\begin{equation}
 \int_{r \in U} |E(z, 1/2 + iT + ir)|^2 dr \ll N^{-2} T^{\varepsilon} \int_{r \in U} |A_N(T+r, T+t_U)|^2 |E(z, 1/2 + iT+ir)|^2 dr.
\end{equation}
By \eqref{eq:IwaniecSarnakBound}, with $\alpha_n = w(n/N) \tau_{i(T+t_U)}(n)$, this is in turn
\begin{equation}
 \ll  (NT)^{\varepsilon} \Big(\frac{T}{N} + T^{1/2} (N + N^{1/2} y) \Big).
\end{equation}
If $1 \ll y \ll T^{1/8}$, we choose $N$ as $T^{1/4}$, which in all gives
\begin{equation}
 |E(z, 1/2 + iT)| \ll T^{3/8 + \varepsilon}.
\end{equation}
If $T^{1/8} \ll y \ll T^{1/6}$, 
we set $N = y^{-2/3} T^{1/3}$, giving
\begin{equation}
\label{eq:mediumT}
 |E(z, 1/2 + iT)| \ll y^{1/3} T^{1/3 + \varepsilon}.
\end{equation}
The bound \eqref{eq:FourierExpansionBound2} is superior to \eqref{eq:mediumT} for $y \gg T^{1/6}$.

\end{document}